\documentclass[12pt,a4paper]{amsart}

\usepackage{amsfonts,amscd}
\usepackage{amsmath,amssymb,bbm}
\usepackage[mathcal]{euscript}

\usepackage[all]{xy}

\newtheorem{theorem}{Theorem}[section]
\newtheorem{prop}[theorem]{Proposition}
\newtheorem{coro}[theorem]{Corollary}
\newtheorem{lemma}[theorem]{Lemma}
\newtheorem{theo}[theorem]{Theorem}
\newtheorem*{conj}{Conjecture}

\theoremstyle{definition}
\newtheorem{exam}[theorem]{Example}
\newtheorem{defin}[theorem]{Definition}
\newtheorem{remk}[theorem]{Remark}

\numberwithin{equation}{section}

\def \Sum{\mathop{\textstyle{\sum}}}

\def\mN{\mathbbm N}

\def\Id{\mathbbm I}
\def\F{\mathfrak f}
\def\f{\mathfrak f}

\def\La{\Lambda}    \def\la{\lambda}
\def\kC{\mathcal C}
\def\de{\delta}	\def\al{\alpha}
\def\ga{\gamma}

\def\kL{\mathfrak L}
\def\kR{\mathfrak R}
\def\kA{\mathfrak A}  
\def\kB{\mathfrak B}  
\def\wA{\widetilde{\kA}}
\def \bV{\overline{V}}
\def\kF{\mathcal F}

\def\rB{\mathrm B}	\def\rD{\mathrm D}

\def\Br{\mathop\mathsf{Br}\nolimits} 
 
\def\add{\mathop\mathsf{add}\nolimits}
\def\gl{\mathop\mathfrak{gl}\nolimits}
\def\GL{\mathop\mathsf{GL}\nolimits}
\def\El{\mathop\mathsf{El}\nolimits}
\def\BT{{\sf BT\hskip-1pt}} 

\def\bR{\mathbf R}

\def\emb{\hookrightarrow}
\def\dar{\hskip-2pt\xymatrix@C=1.0em{\ar@{.*{\dir2{>}}}[r]&}\hskip-2pt}
\def\xarr{\xrightarrow}

\def\set#1{\left\{\,#1\,\right\}}
\def\setsuch#1#2{\left\{\,#1\mid #2\,\right\}}
\def\mtr#1{\begin{pmatrix} #1 \end{pmatrix}}

\def\pp{\partial} \def\td {\tilde{\pp}}
\def\tb{\tilde{b}}
\def\tc{\tilde{c}} \def\x{\tilde{x}}
\def\hb{\hat{\mathfrak b}} \def\y{\tilde y}

\def \ci {~{\scriptstyle{\circ}}~}
\def\ob{\mathop\mathrm{Ob}\nolimits}

\def \aa {\mathfrak a}
\def\bb{\mathfrak b}
\def\bbb{\tilde{\mathfrak b}}

\def \ww {\mbox{\boldmath$\omega$}}
\def\*{\otimes} \def\+{\oplus}
\def\0{\emptyset}  \def\sm{\setminus}

\def \Q{\widetilde Q}

\def\ko{\mathbbm k}
\def\dd{\mathbbm d}

\def\e{\varepsilon} \def\om{\omega}

\def\kmod{\mbox{-}\mathsf{mod}}
\def\bmod{\mbox{-}\mathsf{Mod}}

\def\hom{\mathop\mathrm{Hom}\nolimits}

\def\iff{if and only if }
\def\sot{solid-triangular}

\usepackage{hyperref}

\begin{document}


\author{Lesya Bodnarchuk \and Yuriy Drozd}
\address{Institut des Hautes \'Etudes Scientifiques}
\email{lesya\_bod@ihes.fr}
\urladdr{http://www.mathematik.uni-kl.de/$\sim$bodnarchuk}

\address{Institute of Mathematics, National Academy of Sciences of Ukraine}
\email{drozd@imath.kiev.ua}
\urladdr{http://www.imath.kiev.ua/$\sim$drozd}

\thanks{This work was started during the stay of the first author at
  the Mathematische Forschungsinstitut Oberwolfach as a Leibniz fellow
  and accomplished when both authors were there under the 
 ``Research in  Pairs'' program. We kindly thank the
 Mathematische Forschungsinstitut Oberwolfach for hospitality.
The second author was also partially supported by INTAS
 Grant 06-1000017-9093.}

\title[Brick-tame matrix problems]{One class of wild but brick-tame \\ matrix problems}

\dedicatory{To the memory of A.V.\,Roiter}

\begin{abstract}  
We present a class of wild matrix problems (representations of boxes),
 which are ``\emph{brick-tame},'' 
i.e. only have one-parameter families of \emph{bricks} (representations
with trivial endomorphism algebra). This class includes several boxes
that arise in study of simple vector bundles on degenerations
of elliptic curves, as well as those arising from the coadjoint action of
some linear groups.
\end{abstract}

\maketitle


\tableofcontents

\section{Introduction}
\label{introduction}
Tame--wild dichotomy theorem asserts that any finitely dimensional algebra or a
Roiter box is either tame or wild, i.e. either indecomposable representations of
 any fixed vector dimension form at most finitely many one-parameter families or their description
 contains that of representations of any finitely generated algebra
 \cite{dr79,CB88,dr01}. In the latter (``wild'') case there is no chance to get a
 more or less reasonable classification of \emph{all} representations. Nevertheless,
 there are some wild algebras and boxes, 
 where
 one can get a good description
 of the so called ``\emph{bricks}'', i.e. representations with only scalar endomorphisms.
 Such algebras and boxes appear, for instance, in the theory of unitary representations
 of Lie groups \cite{dr92,bdf} and 
 in the study of vector bundles on degenerations of elliptic curves
 \cite{bd03,bod}.

 In this paper we consider a rather wide class of boxes (called $\BT$-boxes),
 which, though being wild,
 behave well under the so called ``small reduction'' in the sense of \cite{dr92}.
 It implies that the set of bricks  
 of any fixed vector dimension is either empty or form one one-parameter family.
 This class of boxes includes, in particular, the boxes that have appeared in 
 our study of vector bundles on Kodaira fibers in \cite{bd03, bod},
 so that in these cases the bricks correspond
 to \emph{simple vector bundles}. 
 Thus $\BT$-boxes play the key role 
 in the classification of simple vector bundles 
 on Kodaira fibers in the same way as the ``\emph{bunches of chains}'' do 
 in the description of \emph{all}
 vector bundles on Kodaira cycles \cite{drgr,surv}, 
and this paper gives the representation-theoretic background
for such applications. 
 
The following conjecture (due to Claus Ringel) provides
another motivation for studying such sorts of boxes:
 
 \begin{conj}
 Let $\kA$ be a finite dimensional algebra or  a Roiter box. 
 Then either the bricks over $\kA$
  form at most one-parameter families 
  in every fixed vector dimension, 
  or there is a fully faithful exact
 functor $\La\kmod\to\kA\kmod$ for every finitely generated $\ko$-algebra $\La$.
\emph{(In this case they say that $\kA$ is \emph{fully wild}.)}
 \end{conj}

\medskip
 Recall the general method to study representations of boxes,
especially effective for tame ones. The idea can be explained as follows.
For a given class of representations $\kC$ one 
constructs a reduction morphism $\F:\kA\to\wA$ replacing the box $\kA$
by a new one $\wA$ such that the induced functor
$\F^*:\wA\kmod\to \kA\kmod$ is fully faithful and its image contains
all representations from $\kC$.
Moreover, for representations $M\in\kC$ and $\widetilde{M}\in \wA\kmod$
such that $M= \F^*(\widetilde{M})$ one has
$\|\widetilde M\| < \| M\|$, where $\|M\|$ is the \emph{norm} of $M$
 defined in Subsection \ref{subsec 2.1}. Proceeding this way, we construct a morphism of boxes
$
\F=\f_m \f_{m-1}\dots   \f_1: 
\kA=\kA_1\to \kA_2\to\dots\to\kA_m,
$
such that $\kC$ is contained in the image of $\f^*$ and
$\kA_m=(A,V)$ is a \emph{minimal box}, i.e. such that the category $A$
is a direct product of several copies of the field $\ko$ and \emph{rational algebras} $\bR_i$,
i.e. localizations $\bR_i=\ko[t,f_i^{-1}]$ of the polynomial algebra by 
nonzero polynomials $f_i$.
Indecomposable $\bR_i$-modules are \emph{Jordan cells}
$J_r(\lambda)=\bR_i/(t-\lambda)^r$, where $\lambda\in \ko\sm \{\hbox{roots of } f_i\}.$
Thus, all indecomposable modules $M\in\kC$ are of the form $M = \F^*(J_k(\lambda)).$
For instance, the proof of the tame--wild dichotomy is just constructing, for
 any non-wild box $\kA$ and any vector dimension 
$\dd$, a reduction $\F$ like above, where $\kC$ is the set of all
representations of vector dimensions $\dd'\le\dd$.

The following example shows how this method can be applied for bricks.
 \begin{exam}
 \label{exam1}
  Consider the box $\kA$ given, as it is explained below, by the differential biquiver
 \begin{equation}\label{figsmall}
 \xymatrix @ -1pc
    {
1\ar@(ul,dl)_{a_{1}}\ar@{..>}@/_/[rrr]_{v}
&&& 2\ar@(ru,rd)^{a_2}\ar@/_/[lll]_{b}
} \qquad
  \pp(a_1)= bv,\ \pp(a_2)= -vb.
 \end{equation}
 (We do not write the zero differentials.)
 It is known to be wild (see the proof of Theorem~\ref{theo_main}). 
 One of the standard reduction steps, the reduction of the minimal
 edge $b$ (see Subsection~\ref{subsec 2.2} or \cite{dr01}) induces 
 a morphism of boxes $\f:\kA\to\wA$, where the 
 box $\wA$ is given by the differential biquiver 
\begin{equation}
 \label{figbig} \begin{split}
& \xymatrix @ -1pc
    {
1\ar@(ul,dl)_{a_{1}}
\ar@{..>}@/_/[rrr]_{\eta}
\ar@{..>}@/^+25pt/[rrrrrr]_{v}
&&& 
0\ar@(dr,dl)^{a_0}
\ar@{..>}@/_/[rrr]_{\xi}
\ar@/_/[lll]_{b_1}
&&&
2 \ar@(ru,rd)^{a_{2}}
\ar@/_/[lll]_{b_2}
\\
&&& &&&
} \\\
& \pp(a_1) = b_1\eta,\ \pp(a_2) = -\xi b_2,\
\pp(a_0) =- \eta b_1 +b_2\xi. 
 \end{split}
 \end{equation}   
 In this case $\f^*$ is an equivalence, hence, $\wA$ is also wild.
 Nevertheless, since $v$ does not occur in any differential, if a representation
 $M$ of $\wA$ is a brick, either $M_1=0$ or $M_2=0$. But if we cut from $\wA$
 one of the vertices $1$ or $2$, we obtain the same box $\kA$. 
 Thus, if we only consider bricks, the box $\kA$ is ``\emph{self-reproducing}.''
 It easily implies a description of bricks \cite{bd03}. In particular, 
 if there are bricks of some
 vector dimension $\dd$, they form one family parameterized by the elements $\lambda\in\ko$.

 Actually, this procedure is a partial case of ``small reduction''.
 \end{exam}
 
 The paper is organized as follows.
 In Section~\ref{sect2} we fix notations and recall some results 
 concerning representations of boxes and reduction algorithm.
 In Section~\ref{brick} we consider bricks and define \emph{brick-tame} boxes. 
 In Section~\ref{sec Type} we generalize Example~\ref{exam1} introducing \BT-boxes and 
 prove their main property, Theorem~\ref{theo_main}, which claims
 that the boxes of this class are always brick-tame
 despite being wild (except some trivial cases). 
 Section~\ref{coadj} is devoted to a special class of \BT-boxes arising from the adjoint
 action of linear groups over finite dimensional algebras on the dual
 spaces of their Lie algebras.
 

\section{Preliminaries}
\label{sect2}

\subsection{Boxes}
Let $\ko$ be an algebraically closed field. Recall that a $\ko$\emph{-category} is a category
 $A$ such that all morphism sets $A(X,Y)$ are vector spaces over $\ko$, while the
 multiplication of morphisms is $\ko$-bilinear. In what follows we
 only consider $\ko$-categories and identify $\ko$-algebras with
 $\ko$-categories with a unique object.
A tuple $\kA=(A,V,\e, \mu)$ is called a \emph{box} if
$A$ is a $\ko$-category and 
$V$ is a \emph{coalgebra} over $A,$ 
that is an $A$-$A$-bimodule with $A$-homomorphisms
$\e:V\to A$ (\emph{counit}) and $\mu:V\to V\otimes_A V$ 
(\emph{comultiplication}) such that 
$$        (id_V\otimes \mu)\ci \mu =  (\mu\otimes id_V)\ci \mu\ 
        \hbox{ and }\
  (id_V\otimes \e)\ci \mu =  (\e\otimes id_V)\ci \mu=id_V $$
 (under the natural identification of $A\otimes_AV$ and $V\otimes_AA$
 with $V$). Note that any $\ko$-category $A$ can be considered as a box
 if we set $V=A$ as a bimodule over itself, $\e=id_A$ and $\mu$ being
 the identification $A\otimes_AA=A$. This box is called \emph{principal}
 over the category $A$ and is usually identified with this category. 

\medskip
A box $\kA=(A,V)$ is called \emph{free} if $A$ is the path category of a quiver
 (oriented graph) and the \emph{kernel} of the box $\bV=\ker(\e)$ is a \emph{free
$A$-bimodule}, i.e. a direct sum of bimodules of the type $A_{ij}=A1_i\otimes 1_jA$,
 where $1_i$ denotes the empty path at the vertex $i$ (it is a primitive idempotent of $A$).
 We always suppose that the set of vertices of the quiver is $I=\{1,2,\dots,n\}$ and
 denote by $Q_0$ its set of arrows, which we call the \emph{solid arrows} of the box
 $\kA$. Moreover, we also consider the set of \emph{dotted arrows} $Q_1$, 
 where the number of dotted arrows from $j$ to $i$ (denoted as $v:
 j\dar i$) equals the number of summands 
 isomorphic to $A_{ij}$ in the kernel $\bV.$ In other words,
 the arrows of $Q_1$ are in one-to-one
 correspondence with the \emph{free generators} of the kernel $\bV$,
 i.e. those coming from the natural generators $1_i\otimes 1_j$
  of $A_{ij}$, and we usually identify them.
 Thus we obtain a \emph{biquiver}
 $Q=Q_\kA=(I,Q_0,Q_1)$ of the box $\kA$. 
 If $p$ is a path in the biquiver $Q$, its \emph{degree} $|\,p\,|$ is defined as
 the number of dotted arrows occurring in $p$. Thus the path category $\ko Q$
 becomes a graded category. We call a free normal box
 \emph{solid-connected} if the solid part $(I,Q_0)$ of its biquiver is
 connected (as a graph).

 The box $\kA=(A,V)$ is called \emph{normal} (or \emph{group-like}) if there are elements
 $\om_i\in V(i,i)$ such that $\e(\om_i)=1_i$ and $\mu(\om_i)=\om_i\otimes \om_i$
 for every $i\in\ob A$.
 The set $\ww=\setsuch{\om_i}{i\in \ob A}$ is called a \emph{normal section} of the box $\kA$.
 Given a normal section, the \emph{differential} $\partial$ of the box $\kA$
 is defined for a solid arrow $a:j\to i$ as $\om_ia-a\om_j$ (it belongs to $\bV$)
 and for a dotted arrow $v:j\dar i$ as $\mu(v)-v\otimes\om_j-\om_i\otimes v$
 (it belongs to $\bV\otimes_A\bV$). This differential extends to a derivation
 of the graded category $\ko Q$, i.e. to a linear map $\partial:\ko Q\to\ko Q$
 of degree $1$ such that $\partial^2=0$ and the \emph{Leibniz rule} holds:
 \[
  \partial(xy)=\partial(x)y+(-1)^{|x|}x\partial(y) .
 \]
 The pair $(Q,\partial)$ is called the \emph{differential biquiver} of the
 box $\kA$. It completely determines the free normal box $\kA$.
 
 A differential biquiver $(Q,\partial)$ is called
 \emph{\sot} (respectively, triangular) if there is a map $h:Q_0\to\mN$  
 (respectively $h:Q_0\cup Q_1\to\mN$) such that, for every arrow $a\in Q_0$, 
 (respectively, $a\in Q_0\cup Q_1$) its differential $\partial(a)$ only 
 contains solid arrows (respectively, arrows) $b$ with $h(b)<h(a)$
 (for instance, it never contains $a$ itself).  We call the function $h$ the 
\emph{triangulation for the differential biquiver $(Q,\pp)$}.  Certainly,
 the existence of triangulation can depend on the choice of free generators.
 A free normal box $\kA=(A,V)$ is called \emph{\sot} (respectively, \emph{triangular},
 or a \emph{Roiter box}) if there is a set of free generators for $A$ and $V$
 such that the resulting differential biquiver is \sot\ (respectively, triangular).

\subsection{Representations of boxes}
Let $\kA=(A,V)$ be a box.
The category $\kA\bmod$  
of \emph{$\kA$-modules}, or \emph{representations of $\kA$}, is
defined as follows.
\begin{itemize}
  \item  Its \emph{objects} are just $A$-modules.
  \item  A \emph{morphism} $S:M\to N$ between 
    two representations $M$ and $N$ 
    is a homomorphism of $A$-modules $V\*_AM\to N$.
  \item  The product $S'\ci S$ of two morphisms 
$S:M\to N$ and $S': N\to L$ is defined as the composition
 \[  \qquad
  S'(1\*S)(\mu\*1): V\*_AM\to V\*_AV\*_AM\to V\*_AN  \to L.
 \]
\end{itemize}
 One easily sees that if $\kA$ is the principal box over an algebra
$A$, the category of $\kA$-modules can be identified with that of
 $A$-modules, and we always do so.

If $\kA$ is a normal free box, the category of $\kA$-modules can be
 described in terms of its differential biquiver $(Q,\pp)$. Namely:
\begin{itemize}
   \item A representation $M$ of $\kA$ is given by two sets:
    \[
      \qquad\setsuch{M_i}{i\in I}\
      \text{ and }\
     \setsuch{M(a):M_i \to M_j}{a\in Q_0,\ a:i\to j},
    \]
    where $M_i$ are vector spaces and $M(a)$ are linear maps.
   \item  
    A morphism $M\to N$ is given by the set of linear maps 
       \[  \qquad
\setsuch{S_i:M_i\to N_i}{i\in I}\cup\setsuch{ S(v):M_i\to N_j}{v\in Q_1,\ v: i\dar j},
 \]
 where $S_i(x)=S(\om_i\*x)$ and $S_v(x)=S(v\*x)$ for $x\in M_i$,
 such that for any solid arrow $a:i\to j$ the following relation holds: 
      \[
                S_j M(a) - N(a)S_i =S(\pp(a))=\Sum \la N(p')S(u)M(p),
        \]
 if $\pp(a)=\sum \la p'up$, where $\la\in\ko$, $u\in Q_1$ and $p,p'$ are some 
 solid paths in $Q$.
\item 
 The components of the product $T=S'\ci S$ are defined as follows:
 \begin{align*} 
  T_i&=S'_i S_i,\\  \qquad
  T(v)&=S'_j S(v)+S'(v) S_j 
  + \Sum \la L(p_1) S'(u') N(p_2) S(u) M(p_3),
 \end{align*}
 if $v:i\dar j,\ \pp(v)=\sum \la p_1u'p_2up_3$, where $\la\in\ko$,
 $u,u'\in Q_1$ and $p_1,p_2,p_3$ are some solid paths.
\end{itemize}

The following lemma expresses the main properties of Roiter boxes.

\begin{lemma}\label{roi}\cite{kr,Ro2}
Let $\kA$ be a Roiter box and $M,N\in\kA\bmod$.
 \begin{enumerate}
   \item  A morphism $S:M\to N$ is an isomorphism \iff so are all maps $S_i$.
  \item  If $S:M\to M$ is an idempotent, there is a representation $N$ such that $S$
 factors as $S=S_1S_2$, where $S_1:N\to M,\ S_2:M\to N$ and $S_2S_1=\Id_N$
 (the identity map of $N$).
 \end{enumerate}
\end{lemma}

 In other words, all idempotents in the category $\kA\bmod$ split, i.e. it is
 \emph{fully additive} (or \emph{Karoubian}).
   Note that this lemma does not hold for arbitrary \sot{} boxes. We call a
 free normal \sot{} box $\kA$ \emph{layered} (by Crawley-Boevey \cite{CB88})
 if the statement (1) of Lemma~\ref{roi} holds for its representations.
 (In fact, it is a specification of \cite[Definition~3.6]{CB88} for
 the case of free boxes.)

 From now on we only consider normal free
 boxes  $\kA=(A,V)$ such that in the corresponding biquiver $Q=(I,Q_0,Q_1)$ all sets
 $I,Q_0,Q_1$ are \emph{finite}. We call a module $M\in\kA\bmod$ \emph{finite
 dimensional} if all spaces $M(i)\ (i\in I)$ are finite dimensional, and denote by
 $\kA\kmod$ the full subcategory of $\kA\bmod$ consisting of finite dimensional
 modules. Then all spaces $\hom_\kA(M,N)$ are also finite dimensional,
 therefore, if $\kA$ is a Roiter box, 
 $\kA\kmod$ is a \emph{Krull--Schmidt} category, i.e. a fully additive category with
 unique decomposition of objects into direct sums of indecomposable ones. 

\subsection{Base change Lemma.}
\label{subsec 2.1}
Recall that the \emph{vector dimension} of a representation $M\in\kA\kmod$
is a tuple $\dd(M)=(d_1,\dots, d_n)\in \mN^n,$
where $d_i= \dim_{\ko}(M_i).$ The \emph{norm} of $M$ is defined as
$\|M\|= \Sum_{i,j} q_{ij}d_i d_j$, where $q_{ij}$ is the number of solid arrows
 $i\to j$. If we choose bases in all spaces $M_i$, then $\|M\|$ is just the numbers of
 coefficients in all matrices defining the maps $M(a)$, where $a$ runs through $Q_0$. 
 Note that it coincides with the negative part of the \emph{Tits form} of the box $\kA$ 
 as defined, for instance, in \cite{dr01}.

 Now we explain the usual procedures that are the base of
 the reduction algorithm mentioned
 in Introduction. The proofs of the statements can be found, 
for example, in \cite{dr01} .

Let $\kA=(A,V)$ and $\kB=(B,W)$ be some boxes. A \emph{morphism}
$\F=(f_0,f_1): \kA\to\kB$ consists of a functor 
$f_0:A\to B$ and a morphism of $A$-bimodules $f_1:V\to W$
such that
 \[
  \e(f_1(v))=f_0(\e(v))\ \text{ and }\ \mu(f_1(v))=f_2(\mu(v)),
 \]
 where $W$ is considered as an $A$-bimodule using the functor $f_0$%
 \footnote{~It means that $W(i,j)=W(f_0(i),f_0(j))$
 for $i,j\in\ob A$ and $a'xa=f_0(a')xf(a)$ for $x\in W(i,j),
 \ a:i'\to i,\ a':j\to j'$,},
 and $f_2:V\*_AV\to W\*_BW$ is the composition 
 \[
  V\*_AV \xarr{f_1\*f_1} W\*_AW \xarr \nu W\*_BW, 
 \]
 $\nu$ being the natural surjection. Here (and later on) we denote by $\e$ and $\mu$
 the counit and comultiplication in all boxes that we consider (if it cannot lead to
 misunderstanding). Such a morphism $\F $ induces a functor $\F^*:\kB\kmod\to\kA\kmod$,
 where $\F^*M$ is just the composition $M\ci f_0$ (or, the same, we consider the $B$-module
 $M$ as $A$-module using $f_0$) and, for $S\in\hom_\kB(M,N),$
 $\F^*S$  is the composition
 \[
    V\*_AM \xarr{f_1\*1} W\*_AM \xarr{\nu} W\*_BM \xarr S N,  
 \]
 $\nu$ being again the natural surjection.

 This functor is especially useful in the following situation.
 Let $\kA=(A,V)$ be a box, $f:A\to B$ be a functor. Set $\kA^f=(B,W)$,
 where $W=B\*_AV\*_AB$. It becomes a box under naturally defined counit and
 comultiplication, and the pair $\f=(f,f_1)$, where $f_1:V\to W$ is the natural map,
 is a morphism of boxes. The following ``Base Change Lemma'' is the
 most important tool in constructing reduction algorithms. 

  \begin{lemma}[Base Change]\label{change}
  If $\,\kB=\kA^f$ for a functor $f:A\to B$, then the functor $\F^*:
 \kA\kmod\to\kB\kmod$ defined above is fully faithful.
 \end{lemma}

 This lemma is mostly used in the following situation. Let $\kA'=(A',V')$ be a
 \emph{subbox} of the box $\kA=(A,V)$. It means that $A'$ is a subcategory
 of $A$ and $V'$ is an $A'$-subbimodule of $V$ such that $\e(a)\in V'$ for
 all $a\in A'$ and $\mu(v)\in\nu(V'\*_{A'}V')$ for all $v\in V'$, where again
 $\nu$ is the natural surjection $V\*_{A'}V\to V\*_AV$. If $\kA$ is a free normal
 box with the differential biquiver $(Q,\pp)$ and $Q'$ is a sub-biquiver of $Q$ such that, for every
 arrow (solid or dotted) $a\in Q'$, its differential $\pp(a)$ only contains arrows from
 $Q'$, the box $\kA'$ defined by the biquiver $Q'$ and the differential $\pp|_{Q'}$ 
 is a subbox of $\kA$. In this case we say that $\kA'$ is a \emph{Roiter subbox} 
 of $\kA$. Lemma~\ref{change}, together with the universal property of  push-down
 (amalgamation), imply the following fact.

 \begin{coro}\label{amalgama}
 Suppose that $\kA'=(A',V')$ is a subbox of the box $\kA=(A,V)$ and a functor
 $f':A'\to B'$ is given. Let $B$ be the \emph{amalgamation} of the categories
 $A$ and $B'$ over $A'$, i.e. the push-down
 \[
    \begin{CD}
       A' @>\iota>> A \\
       @Vf'VV     @VVfV \\
        B' @>>> B\,,
 \end{CD}
 \]
 where $\iota$ denotes the embedding $A'\emb A$. Then the image of the functor 
 $\F^*:\kA^f\kmod\to\kA\kmod$ consists of the modules $M$ whose restrictions
 $M|_{A'}$ factor through $\F'$. In particular, if every $\kA'$-module
 $M'$ is isomorphic (in $\kA'\kmod$) to a module that factors through $f'$, the
 image of $\F^*$ is dense, so $\F^*$ is an equivalence of categories.
 \end{coro}


Note that, if the subbox $\kA'$ is representation-finite, there is
a natural functor $f':A' \to \add B'$, the additive hull of a
discrete category $B'$\,%
\footnote{~The category $B'$ is called \emph{discrete} if
for $i,j\in \ob(B')$ we have $B'(i,i)=\ko$  and $B'(i,j)= 0$ if $ i\neq j.$}
whose objects are isomorphism classes of indecomposable $\kA'$-modules, and
any representation of $\kA'$ is isomorphic to one that factors through $f'$.
This morphism is called the \emph{semisimple approximation}
and was introduces in \cite{Ausl} for algebras of finite type.
(Its extension to representation-finite boxes is obvious.)

 We shall use it in the following two situations:
 \emph{regularization} and \emph{minimal edge reduction}, first
 introduced in \cite{kr}. We recall some details (see, for instance,
 \cite{dr01}). In both cases $\kA$ is a free normal box with the
 differential biquiver $(I,Q_0,Q_1,\pp)$ and $\kA'$ is a
 Roiter subbox of $\kA$.

\subsection{Regularization}
\label{regul}
The subbox $\kA'$ has the differential biquiver 
$$\xymatrix @ -1pc
    {2\ar[rr]_{b}\ar@{..>}@/^/[rr]^{v}
&& 1,}\qquad \pp(b)=v.$$
 (In this case they say that the solid arrow $b$ is
 \emph{superfluous}). Then $B'$ is the discrete category with two
 objects, which we also denote by $1$ and $2$,
 $f'(1)=1,\,f'(2)=2,\,f'(b)=0$. The box $\kA^f$ is obtained from $\kA$
 just by deleting the arrows $b$ and $v$ from the biquiver as well as
 omitting all terms containing these arrows in all formulae for the
 differential. Note that the case when the vertices $1$ and $2$
 coincide is also possible. We denote the box $\kA^f$ by $\kA^b$. 
 Evidently, this box is \sot\ (respectively, layered or a Roiter box)
 if so is $\kA$.

\subsection{Minimal edge reduction}
\label{subsec 2.2}
 The subbox $\kA'$ has the differential biquiver $2\xarr b 1,\
 \pp(b)=0$, where $1$ and $2$ are different vertices. 
 (In this case they say that $b$ is a \emph{minimal edge}.) Take for
 $B'$ the additive hull of the trivial category with $3$ objects
 $0,1,2$, where $0$ is a new symbol, and set
 $f'(1)=1\+0,\,f'(2)=2\+0,\,f'(b)=
 \mtr{0&0\\0&1}
 $. Then every
 $A'$-module factors through $f'$, so Corollary~\ref{amalgama} can
 be applied and $\F^*$ is an equivalence of categories. One can check
 (cf. \cite{kr} or \cite{dr01}) that the box $\kA^f$ can be
 identified with the additive hull of the box $\kB$ with the
 differential biquiver $(\Q,\td)$ defined as follows:
 \begin{itemize}\label{formule}
  \item  The set of vertices of $\Q$ is $I\cup\{0\}$.
  \item  The set of arrows of $\Q$ consists of:
    \begin{itemize}
     \item  The arrows $x:i\to j$, where $\set{i,j}\cap\set{1,2}=\0$.
     \item  For each arrow $x:i\to j$ (or $j\to i$), where
     $i\in\set{1,2}$ and $j\notin\set{1,2}$, we have two arrows $x_i:i\to j$
     and $x_0:0\to j$ (respectively, $x_i:j\to i$ and
     $x_0:j\to0$). Then we set $\F(x)=\mtr{x_i&x_0}$ \big(respectively,
     $\F(x)=\mtr{x_i\\x_0}$\big). 
     \item  For each arrow $x:j\to i$, where both $i,j\in\set{1,2}$
     and $x\ne b$, we have four arrows $x_{kl}:l\to k$, where
     $k\in\set{i,0},\,l\in\set{j,0}$. Then we set
     $\F(x)=\mtr{x_{ij}&x_{i0}\\x_{0j}&x_{00}}$.
     \item  Two new dotted arrows $\xi:0\dar1$ and $\eta:2\dar0$.
    \end{itemize}
    Certainly, the arrows arising from $x$ are solid or dotted
    respectively to the sort of $x$. We also set $\F(b)=f'(b)$,
    $\F(\om_i)=\om_i$ if $i\notin\set{1,2}$,
    $\F(\om_1)=\mtr{\om_1&0\\ \eta&\om_0}$,
    $\F(\om_2)=\mtr{\om_2&\xi\\0&\om_0}$ and extend the map $\F$ 
    naturally to all elements from $A$ and $V$.
  \item  The differential $\td$ is obtained from the rules 
    \begin{align*} \hskip2em
      & \F(\om_j)\F(a)-\F(a)\F(\om_i)=\td(\F(x)) \text{ for } a:i\to j,\\
      & \mu(\F(v))-\F(v)\*\F(\om_i)-\F(\om_j)\*\F(v)=\td(\F(v)) \text{ for
      } v:i\dar j, 
    \end{align*}
    where all products, as well as tensor products,
    are calculated by usual matrix rules, while $\td$ and $\mu$ are
    applied to matrices component-wise. 
 \end{itemize}
 Therefore, $\kA\kmod\simeq\kA^\F\kmod\simeq\kB\kmod$. We denote the
 box $\kB$ by $\kA^b$. Again this new box is \sot\ (respectively, 
 layered or a Roiter box) if so is $\kA$.
 
The following theorem summarizes the above considerations.  

 \begin{theorem}[Kleiner--Roiter]\label{K-R}
   Let $\kA$ be a free normal box, $b:2\to 1$ be either a superfluous arrow or 
   a minimal edge of its differential biquiver. Then there is a free
   normal box $\kA^b$ and an
   equivalence of module categories $\F^b:\kA^b\kmod\to\kA\kmod$ such
   that $\|\F^b(M)\|<\|M\|$ whenever $M\simeq\F^b(N)$ is such that both
   $M(1)\ne0$ and $M(2)\ne0$. Moreover, the box $\kA^b$ is \sot\ (respectively, 
   layered or a Roiter box) if so is $\kA$.
 \end{theorem}

 We also often need to delete vertices from a free normal box $\kA$. If
 $i$ is a vertex of the biquiver of $\kA$, we denote by $\kA^i$ the
 box that  is obtained from $\kA$ by deleting the vertex $i$ from its
 biquiver and omitting all terms in differentials containing arrows 
 starting or ending at $i$. Obviously, $\kA^i\kmod$ is identified with the 
 full subcategory of $\kA\kmod$ consisting of all modules $M$ with $M_i=0$.

\section{Bricks} 
\label{brick}

\begin{defin}
 A representation of a box (in particular, of an algebra)
 is called a \emph{brick} if it admits no non-scalar  endomorphisms. 
The full subcategory of bricks of $\kA\kmod$ is denoted by
 $\Br(\kA)$. We also denote by $\Br(\dd,\kA)$ the set of isomorphism
 classes of bricks of vector dimension $\dd$.
\end{defin}

 \begin{lemma}
 \label{lemmasplit}
  Let $\kA$ be a normal free box with the differential biquiver
  $(Q,\pp)$ containing  a dotted arrow $u:i\dar j$ that does not occur
  in the differential of any solid arrow. If $M\in\Br(\kA)$, then
  either $M_i=0$ or $M_j=0$. Thus $\Br(\kA)=\Br(\kA^i)\cup\Br(\kA^j)$.
 \end{lemma}
 \begin{proof}
 If both $M_i\ne0$ and $M_j\ne0$, we construct a non-scalar
 endomorphism $S$ of $M$ setting $S_k=0$ for all vertices $k$,
 $S(v)=0$ for all dotted arrows $v\ne u$ and taking for $S(u)$ any
 nonzero linear map $M_i\to M_j$.
 \end{proof}

 We have actually applied this lemma to the box \eqref{figbig} in
 Example~\ref{exam1} of the Introduction. 

 \medskip
Since we are going to study bricks instead of indecomposable
representations, we have to adapt the classical definition of tameness
for our purposes. Recall \cite{dr01} that a \emph{rational family} of
representations of a box $\kA=(A,V)$ is defined as a functor
$\kF:A\to\add\bR$, where $\bR=\ko[t,f(t)^{-1}]$ is a rational
algebra. Note that $\add\bR$ can be identified with the category of
finitely generated projective  $\bR$-modules. The $\bR$-bricks (or,
the same, $\add\bR$-bricks) are just one-dimensional representations of
$\bR$, which we identify with the elements $\la\in\ko$ such that
$f(\la)\ne0$. If for every such $\la$ the $\kA$-module $\kF^*(\la)$ is
a brick and, moreover, $\kF^*(\la)\not\simeq\kF^*(\la')$ for all
$\la\ne\la'$, we say that $\kF$ is a \emph{rational family of
bricks}. We also say that the bricks isomorphic to $\kF^*(\la)$
\emph{belong to the family} $\kF$.

\begin{defin}  
 A box $\kA$ is called  \emph{brick-tame} if for any vector dimension
 $\dd$ there is a finite set $\Sigma$ of rational families of bricks
 such that all $\kA$-bricks of vector dimension $\dd$, except,
 possibly, finitely many of them, belong to one of the families from
 $\Sigma$. (Note that we allow the case when there are only
 finitely many bricks of vector dimension $\dd$.)
\end{defin}  

 Obviously, every tame box is brick-tame, but not vice versa: the box
 \eqref{figsmall} from the Introduction is wild, but brick-tame.

\section {\BT-boxes}
\label{sec Type}

In this section we introduce a special class of brick-tame boxes that
generalizes the boxes from Example~\ref{exam1}.

\begin{defin}
A \sot\ box $\kA$ with the differential biquiver  $(Q,\pp)$ 
is said to be of  \emph{\BT-type}, or a \emph{\BT-box},
if  $Q_0$ contains a set of loops (called \emph{distinguished loops})
$\aa=\setsuch{a_i:i\to i}{i\in I}$ 
and there is an injective map $\tilde{}:\bb=Q_0\sm\aa\to Q_1$, $x\mapsto\x$, 
such that $\x:j\dar i$ if $x:i\to j$ and
\begin{equation}
\label{eq_dist}
 \pp(a_i) = \Sum_{x\in\hb(\cdot, i)}(-1)^{|x|}x\x
\end{equation}
for each distinguished loop $a_i\in \aa$,
where we set $\tilde{\bb}=\{\x\,|\, x\in \bb\}$, $\hb=\bb\cup\bbb$  and 
$\tilde{\tilde{x}}=x$ for each $b\in\bb$.
\end{defin}
Both boxes (\ref{figsmall}) and (\ref{figbig}) from Example \ref{exam1}
are \BT-boxes. The polynomial algebra $\ko[t]$ is also
a \BT-box (having only one vertex and one solid arrow, which is
automatically a distinguished loop).
 The following theorem asserts that the \BT-boxes 
 are brick-tame despite being wild in general.

\begin{theo}\label{theo_main}
Let $\kA$ be a \BT-box.
\begin{enumerate}
 \item  $\kA$ is brick-tame. Moreover, if
 $\Br(\dd,\kA)\ne\0$, all bricks of dimension $\dd$ belong
 to a unique rational family $\kF_\dd:A\to\add\ko[t]$.  
 \item  If $\kA$ is solid-connected, has no superfluous arrows
 and does not coincide with $\ko[t]$, it is wild.
\end{enumerate}
\end{theo}

The claims of the theorem are trivial if $\kA=\ko[t]$.
Moreover, we may suppose that $\kA$ is solid-connected.
The proof of the theorem is based on several lemmas. 
\emph{In all of them $\,\kA=(A,V)$ denotes a solid-connected \ \BT-box 
that does not coincide with $\ko[t]$, $(Q,\pp)$
is its differential biquiver and $h:Q_0\to\mN$ is a
triangularity for this biquiver.} Note that if 
$a$ is a distinguished loop from $\aa$, then $\pp(a)$
contains some solid arrows $b$ with $h(b)<h(a)$.
Hence, if $b$ is an arrow with the minimal value of $h(b)$,
it belongs to $\bb$. We then call $b$ an \emph{$h$-minimal arrow}.

\begin{lemma}\label{lemma0}
 Suppose that $b$ is an $h$-minimal arrow in
 $Q_0$ and $\pp(b)\ne0$. Then one can choose free generators
 of the category $A$ and of the kernel $\bV$ in such a way
 that
 \begin{enumerate}
	\item $\pp(b)=\tc$ for some solid arrow $c$ and $\pp(c)=-\tb+\theta$,
 where $\theta$ does not contain the arrow $\tb$.
  \item $\tb$ does not occur in $\pp(x)$ for any arrow (solid or 
  dotted) $x\notin\{c,\tb\}$.
 \end{enumerate}
 Moreover, with respect to the new generators $\kA$ remains \sot.
 We call $c$ the \emph{partner} of $b$.
 \end{lemma}
 \begin{proof}
 Let $b:i\to j$ (possibly $i=j$). Since $\pp(b)$ cannot contain 
 any solid arrow, we may suppose that $\pp(b)=u+\sigma$, where $u\in Q_1$
 and $\sigma$ is a sum of dotted arrows other than $u$. Then
 \begin{align*}
	\pp(a_i)& =-\tb b+\Sum_{\substack{x\in\hb(\cdot,i)\\x\ne \tb}} 
	(-1)^{|x|}x\x,\\
	\pp^2(a_i)& = -\pp(\tb)b+\underline{\tb u}+\tb\sigma+
	\Sum_{x\ne \tb}\big((-1)^{|x|}\pp(x)\x+x\pp(\x)\big)=0.
\end{align*}
 Since the underlined term must vanish, there must be $x\ne\tb$
 such that $u=\x$ and $\pp(x)=-\tb+\theta$, where $\theta$ does not contain 
 the monomial $\tb$. Therefore, $\pp(b)=\sum_k\x_k$ for some $x_k:j\to i$.
 Let $h(x_1)\le h(x_k)$ for all $k$. Set $c=x_1,\,x'_k=x_k-x_1$
 for $k\ne1$; $\tc=\pp(b),\,\x'_k=\x_k$ for $k\ne1$, $h(x_k')=h(x_k)$. One easily
 sees that after this change of generators the \BT-condition 
 \eqref{eq_dist} as well as the triangularity condition hold,
 but now $\pp(b)=\tc$ and $\pp(c)=-\tb+\theta$,
 as stated in (1). Note also that $\pp(\tc)=\pp^2(b)=0$.
 
 To prove (2) we have to show that it is impossible that 
\begin{equation}\label{imp}
	\pp(y)=q\tb p+\phi
\end{equation}
 for some arrow $y\ne c$ and some paths $p,q$, where $\phi$ 
 does not contain the monomial $q\tb p$. We prove this claim
 using induction on the length $l(q)$ of the path $q$. If $l(q)=0$,
$\pp(y)=\tb p+\phi$. We first suppose that $y\ne\tb$. Then
\begin{align*}
	\pp(a_i)& =-\tb b+c\tc +(-1)^{|y|}y\y +
	\Sum_{\substack{x\in\hb(\cdot,i)\\x\notin\{\tb,c,y\}}}(-1)^{|x|}x\x,\\
	\pp^2(a_i)& = -\pp(\tb)b+\theta c+(-1)^{|y|}(\underline{\tb p\y}+\phi\y)+y\pp(\y)\\
	&+\Sum_{x\notin\{\tb,c,y\}}\big((-1)^{|x|}\pp(x)\x+x\pp(\x)\big)=0.
\end{align*}
 But the underlined term cannot vanish, since no other term both 
 starts with $\tb$ and ends with $\y$. If $y=\tb$, we must omit in these
 equalities all terms with $y$ and $\y$ and replace $\pp(\tb)$ by
 $\tb p+\phi$. Then the term $-\tb pb$ cannot vanish. (This case can
 happen if $\kA$ is only \sot, but not a Roiter box).
 
 Suppose now that \eqref{imp} is impossible if $l(q)=l-1$, but holds 
 for an arrow $y\ne\tb$ and a path $q$ of length $l$.
 (The case $y=\tb$ is handled in the same way, as above.) Then
 \begin{align*}
	\pp^2(a_i)&= -\pp(\tb)b+\theta c+(-1)^{|y|}(\underline{q\tb p\y}+\phi\y)+y\pp(\y)\\
	&+\Sum_{x\notin\{\tb,c,y\}}\big((-1)^{|x|}\pp(x)\x+x\pp(\x)\big)=0.
 \end{align*}
 Therefore, the underlined term must vanish. It is only possible if
 there is an arrow $x$ such that $q=xq'$ and $\pp(\x)$ contains the term
 $q'\tb p\y$. Since $l(q')=l-1$, it is impossible, which accomplishes
 the proof. 

 Just in the same way one proves that the term $\theta$ in $\pp(c)$ cannot
 contain monomials $q\tb p$ other than $\tb$.
 \end{proof}

 Analogous observations can be applied to the case when $\pp(b)=0$.

\begin{lemma}
\label{lemma1}
 If $\pp(b)=0$ for some solid arrow $b$, the arrow $\tb$ does not 
 occur in the differential of any arrow $x$ (solid or dotted).
 \end{lemma}
\begin{proof}
 It practically coincides with the proof of Lemma~\ref{lemma0}, so we 
 omit the details.
\end{proof}

 These lemmas imply several nice properties. 

\begin{coro}
\label{coro1}
 Let $b:i\to j$ be an $h$-minimal arrow with $\pp(b)\ne0$, $c$ be its partner as
 defined in Lemma~\ref{lemma0}, $B=A/\langle b,c\rangle$, $f:A\to B$ be the
 natural surjection and $\kB=\kA^f$. Then $\kB$ is also of \BT-type,
 the functor $\f^*:\kB\kmod\to\kA\kmod$ is an equivalence and, if
 $M_1\ne0,\,M_2\ne0$ and $M\simeq\f^*(N)$, then 
 $\|N\|<\|M\|$. 
 \end{coro}
 \begin{proof}
 Set $v=-\pp(c)$. Obviously, we can replace $\tb$ by $v$ in the set of
 free generators of $\bV$. Then both $b$ and $c$ become superfluous,
 so we can use the regularization procedure of Subsection~\ref{regul}
 for both of them obtaining just the box $\kB$. The images of $b,c,\tc$
 and $v=\phi-\tb$ become zero in $\kB$. Since $\tb$ only occurs
 in differentials of distinguished loops, always in terms $b\tb$ and $\tb b$,
 which disappear in $\kB$, the box $\kB$ is also \sot\ (with the same
 triangulation) and of \BT-type. The statement now follows 
 from Theorem~\ref{K-R}.
 \end{proof}
 
 Let now $\pp(b)=0$. First we show that the case when $b$ is a loop
 actually cannot occur. Recall that we suppose that $\kA$ is 
 solid-connected and does not coincide with $\ko[t]$, hence, $b$ 
 cannot be distinguished.

\begin{coro}
\label{minedge}
 If $b:1\to1$ is a solid loop with $\pp(b)=0$ and $M\in\Br(\kA)$,
  then $M_1=0$.
  \end{coro}
\begin{proof}
 Suppose that $M_1\ne0$. Then 
 \[
  \pp(a_1)=b\tb  -\tb  b+
  \Sum_{x\notin\{b,\tb\}}(-1)^{|x|}x\x.  
  \]
  Set $S_i=0$ for all $i$, $S(v)=0$ for all $v\in Q_1\sm\{\tb\}$ and 
  $S(\tb)=\Id_{M_1}$. Since $\tb$ does not occur in any differential of a
  solid arrow other than $a_1$, $S$ is a non-scalar endomorphism of $M$,
  so $M$ is not am brick.
\end{proof} 

 The minimal edge reduction described in Subsection~\ref{regul} usually
 does not give a \BT-box if the original one was so. Nevertheless, the 
 following result holds.

\begin{lemma}[Self-Reproduction] 
\label{self}
 Let $\kA$ be a \BT-box, $b:2\to 1$ be a minimal edge in $\kA$. Then
 there is a morphism of boxes $\kA^b\to\kB$, which is actually a
 composition of regularizations, such that $\kB$ is also a \BT-box and
 $\Br(\kB)=\Br(\kB^1)\cup\Br(\kB^2)$. Moreover, if $M_1\ne0,\,M_2\ne0$
 and $M\simeq\f^*(N)$, where $\f$ is the composition $\kA\to\kA^b\to\kB$, 
 then $\|N\|<\|M\|$.
 \end{lemma}
\begin{proof}
  We denote $a=a_1,\,c=a_2,\,v=\tb$. Then the rules for the minimal
  edge reduction (see page~\pageref{formule}) result in the 
  biquiver 
  \[
\xymatrix @ -1pc
    {
1\ar@(ul,dl)_{a_{11}}
\ar@<+0.8pc>@{..>}@/^/[rrr]^{\eta}
\ar@<-0.8pc>@{..>}@/_/[rrr]_{v_{01}}
\ar@<1pc>@{..>}@/^+25pt/[rrrrrr]^{v_{21}}
\ar@<-0.3pc>[rrr]_{a_{01}}
&&& 
~~0~~\ar@(ru,ul)_{a_{00}}
\ar@(ld,d)_{c_{00}} \ar@{..>}@(rd,d)^{v_{00}}
\ar@<-0.3pc>[lll]_{a_{10}}
\ar@<+0.8pc>@{..>}@/^/[rrr]^{\xi}
\ar@<-0.8pc>@{..>}@/_/[rrr]_{v_{20}}
\ar@<-0.3pc>[rrr]_{c_{20}}
&&&
2 \ar@(ru,rd)^{c_{22}}
\ar@<-0.3pc>[lll]_{c_{02}}
\\
}
\]
with the differential
  \begin{equation}\label{dif1}
   \begin{split}
    \pp(a_{11})&=a_{10}\eta+\alpha_{11},\\
    \pp(a_{10})&=\alpha_{10},\\
    \pp(a_{01})&=v_{01}+a_{00}\eta-\eta a_{11}+\alpha_{01},\\
    \pp(a_{00})&=v_{00}-\eta a_{10}+\alpha_{00},\\
    \pp(c_{22})&=-\xi c_{02}+\beta_{22},\\
    \pp(c_{20})&=-v_{20}+c_{22}\xi-\xi c_{00}+\beta_{20},\\
    \pp(c_{02})&=\beta_{02},\\
    \pp(c_{00})&=c_{02}\xi-v_{00}+\beta_{00},
   \end{split}
  \end{equation}
where $\alpha_{kl}$ and $\beta_{kl}$ 
are collections of the other terms, which 
do not contain the
arrows $a_{kl},c_{kl},v_{kl}$. Moreover, the terms $\alpha_{kk}$ and
$\beta_{kk}$ are just of the form $\sum_x(-1)^{|x|} x\x$, as in
\eqref{eq_dist}, with $x$ and $\x$ different from $a_{kl}$ and
$c_{kl}$. Set
$\pp(a_{01})=u_{01},$ $\pp(c_{20})=u_{20}$, $\pp(a_{00})=u_{00}.$ One easily
sees that we can replace the generators $v_{kl}$ of the kernel of the
box $\kA^b$ by $u_{kl}$ so that the resulting set of generators remains \sot. 
Then the arrows $a_{01},\,a_{00},\,c_{20}$ become
superfluous. After regularization they disappear, as well as the
dotted arrows $u_{01},\,u_{00},\,u_{20}$, and the formulae \eqref{dif1}
change to: 
 \begin{equation}\label{dif2}
   \begin{split}
    \pp(a_{11})&=a_{10}\eta+\alpha_{11},\\
    \pp(a_{10})&=\alpha_{10},\\
    \pp(c_{22})&=-\xi c_{02}+\beta_{22},\\
    \pp(c_{02})&=\beta_{02},\\
    \pp(c_{00})&=c_{02}\xi-\eta a_{10}+\alpha_{00}+\beta_{00}.
   \end{split}
  \end{equation}
 Therefore, if we set
 $a_1=a_{11},$ $a_2=c_{22},$ $a_0=c_{00},$ $\tilde a_{10}=\eta$ and
 $\tilde c_{02}=\xi$, we see that the resulting box $\kB$ is indeed a
 \BT-box.

 Moreover, the dotted arrow $v_{21}$ does not occur in the differential
 of any solid arrow (since,  by Lemma~\ref{lemma1}, the arrow $v$  
 was not involved in the differentials of arrows from $\bb$). 
 Therefore, by Lemma~\ref{lemmasplit}, $\Br(\kB)=\Br(\kB^1)\cup\Br(\kB^2)$.
 The other statements follow from Theorem~\ref{K-R}.
\end{proof}

 \begin{remk}
 The boxes $\kB^1$ and $\kB^2$ are actually obtained from $\kA$ by the 
 \emph{small reduction} of the minimal edge $b$ as defined in \cite{dr92}.
 \end{remk}

Having  these Lemmas, the proof of the theorem is quite obvious.

\begin{proof}[Proof of Theorem \ref{theo_main}]
 (1)  
Let $\kA$ be a \BT-box with the differential bi\-quiver $(Q,\pp)$
and $\dd$ be a vector dimension such that the set
$\Br(\dd,\kA)$ of bricks 
of vector dimension  $\dd$ is non-empty. Without loss of generality,
we may suppose that $\kA$ is solid-connected and $d_i\ne0$ for all $i$.
Then $Q$ contains no non-distinguished loops with zero differential 
by Corollary~\ref{minedge}. Thus, if $Q$ only has one vertex, $\kA=\ko[t]$ 
and the statement is trivial. Let $b:2\to1$ be an $h$-minimal arrow.
Without loss of generality assume that $d_2\le d_1.$ Consider the box
$\kB$ and the morphism $\F:\kA\to\kB$ constructed in Corollary~\ref{coro1}
 or Lemma~\ref{self}. The box $\kB$ is also of \BT-type, 
 $\f^*:\kB\kmod\to\kA\kmod$ is an equivalence of categories, and 
 $\|N\|<\|M\|$ for any $M\in\kA\kmod$ such that $\dd(M)=\dd$
 and $\f^*(N)\simeq M$. Especially, this inequality holds if
$M\in\Br(\dd,\kA)$. Moreover, then $N$ is also a brick. If $b$ was
superfluous, $\dd(N)=\dd(M)$. If $b$ was a minimal edge, then 
 either $N_1=0$ or $N_2=0$, while $\dim M_1=\dim N_1+\dim N_0$ and 
$\dim M_2=\dim N_2+\dim N_0$. Since $d_2\le d_1$,
it implies that $N_2=0$, $\dim N_0=d_2$ and $\dim
N_1=d_1-d_2$. Hence, the vector dimension of $N$ is uniquely
determined. So we can proceed by induction on $\|M\|$, since,
if $\|M\|=1$ and $\kA$ is solid-connected, it only contains 
one vertex.

\medskip
 (2)
 Let $b:i\to j$ be an $h$-minimal arrow of $\kA$. Then
 $\pp(b)=0$, so the differential biquiver $(Q,\pp)$ of
 $\kA$ has a fragment $(Q',\pp')$ that coincides either with the biquiver
 \eqref{figsmall} from the Introduction or with the biquiver
 \begin{equation}\label{newfig}
  \xymatrix{i \ar@(dl,ul)^{a_i} \ar@(dr,ur)_b \ar@(dl,dr)@{.>}_\tb  } \qquad
    \pp(a_i)=b\tb  -\tb  b,\ \pp(b)=0.
 \end{equation}
Let $\overline\kA$ be the box with the biquiver $(Q',\pp')$, $\overline M$ be an
$\overline\kA$-module and $M$ be the $\kA$-module that coincide with $\overline M$
on $Q'$ and is zero outside it. One easily checks that for two such $\kA$-modules
$M$ and $M'$ we have 
$M\simeq M'$ \iff $\overline M\simeq\overline M'$. Hence, if
$\overline\kA$ is wild, so is also $\kA$. The box \eqref{newfig} is
well-known to be wild \cite{dr79,CB88}. So we only have to prove that the
box \eqref{figsmall} (or, equivalently, \eqref{figbig}) is also
wild. To do it, we use the reduction procedure described above.

In what follows, when drawing the biquiver of a \BT-box, we omit the
distinguished loops and their differentials and do not precise the
names of the dotted arrows $\tb $, since these data can be uniquely
restored. We also do not mention in the list of differentials
the arrows $b$ with $\pp(b)=0$.  Moreover, we usually omit the dotted 
arrows that do  not occur in the differentials. Obviously, omitting 
such dotted arrows does not affect the representation type. In 
particular the differential biquiver \eqref{figsmall} will be presented as
 \[
  \xymatrix{1 \ar@/_/@{.>}[rr]&&2 \ar@/_/[ll]_b}
 \]
and the differential biquiver \eqref{figbig} as
 \[
  \xymatrix{1 \ar@/_/@{.>}[rr]  && 0 \ar@/_/[ll]_{b_1}\ar@/_/@{.>}[rr]
 && 2 \ar@/_/[ll]_{b_2} }
 \]
 (we omit $v$). After the reduction of the minimal edge $b_1$ and
 regularization, we get the \BT-box with the biquiver
 
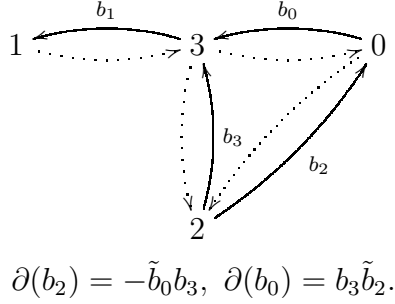
\begin{equation*}
 \begin{split}
  \xymatrix{
    1 \ar@/_/@{.>}[rr] && 3 \ar@/_/[ll]_{b_1} \ar@/_/@{.>}[rr]
    \ar@/_/@{.>}[dd] && 0 \ar@/_/[ll]_{b_0}\ar@/_/@{.>}[ddll] \\ \\
   && 2 \ar@/_/[uurr]_{b_2} \ar@/_/[uu]_{b_3}
    } \\
  \pp(b_2)=-\tb_0b_3,\ \pp(b_0)=b_3\tb_2.
 \end{split}
 \end{equation*}
  Now we reduce the minimal edge $b_3$. After regularization, 
 we get the \BT-box with the biquiver
 
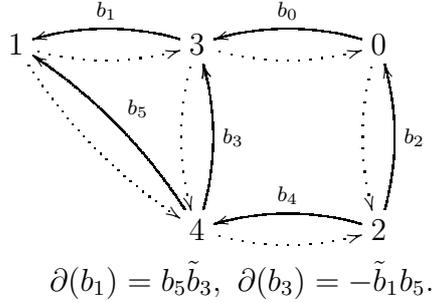
\begin{equation*}
 \begin{split}
  \xymatrix{
    1 \ar@/_/@{.>}[rr] \ar@/_/@{.>}[ddrr] 
    && 3 \ar@/_/[ll]_{b_1} \ar@/_/@{.>}[rr] \ar@/_/@{.>}[dd] 
    && 0 \ar@/_/[ll]_{b_0}\ar@/_/@{.>}[dd] \\ \\
    && 4 \ar@/_/@{.>}[rr] \ar@/_/@{.>}[uu] \ar@/_/[uu]_{b_3}
    \ar@/_/[uull]_{b_5}
    && 2 \ar@/_/[uu]_{b_2} \ar@/_/[ll]_{b_4}
    } \\
  \pp(b_1)=b_5\tb_3,\ \pp(b_3)=-\tb_1b_5.
 \end{split}
 \end{equation*}
 If we factor out the arrow $b_5$, the remaining non-distinguished
 arrows become minimal and form the quiver
 
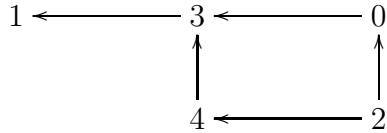
\[
  \xymatrix{
    1 && 3 \ar[ll]  && 0 \ar[ll]  \\
    && 4 \ar[u] && 2 \ar[ll] \ar[u]  }
 \]
which is neither Dynkin nor Euclidean, hence, wild. Therefore, so is
also the box \eqref{figsmall}. It accomplishes the proof.
\end{proof}

\section{\BT-boxes and coadjoint action}
\label{coadj}
 
 A natural class of \BT-boxes arises from linear groups over algebras.
 Recall \cite{dr92,bdf} that a \emph{linear group over an algebra} $\La$ is,
 by definition, the group $\GL(P,\La)$ of automorphisms of a finitely
 generated projective $\La$-module $P$. If $\La$ is finite dimensional
 over a field $\ko$, $\GL(P,\La)$ is a linear group over $\ko$ (Lie group
 if $\ko$ is the field of complex numbers). Its Lie algebra $\gl(P,\La)$
 is just the commutator algebra of the endomorphism algebra of $P$. 
 In the representation theory of 
 linear and Lie groups the coadjoint action of a group on the dual space 
 of its Lie algebra, especially its orbit space, plays an important role. 
 Note that the dual space of $\gl(P,\La)$ is
	\begin{align*}
	\gl^*(P,\La)&=
	\hom_\ko(\hom_\La(P,P),\ko)\simeq \hom_\ko(P^\vee\*_\La P,\ko)\\
	 &\simeq \hom_\La(P,\hom_\ko(P^\vee,\ko))
	 \simeq\hom_\La(P,\hom_\ko(P^\vee\*_\La\La,\ko))\\
	 &\simeq\hom_\La(P,\hom_\La(P^\vee,\La^*))
	 \simeq\hom_\La(P,\La^*\*_\La P),
	\end{align*}
	where $P^\vee=\hom_\La(P,\La)$ and $\La^*=\hom_\ko(\La,\ko)$.

  From now on suppose that the algebra $\La$ is \emph{basic}, i.e. 
  if $1=\sum_{i=1}^ne_i$, where $e_i$ are pairwise orthogonal primitive
  idempotents, $\La e_i\not\simeq\La e_j$ as $\La$-modules for $i\ne j$.
  Since every finite dimensional algebra is Morita-equivalent to a basic
  one, every linear group over a finite dimensional algebra is isomorphic
  to a linear group over a basic algebra.
	If we fix the algebra $\La$ and consider all linear groups
 $\GL(P,\La)$, the description of orbits in all dual spaces $\gl(P,\La)$
 coincides with the ``\emph{bimodule problem},'' namely, the description of
 isomorphism classes in the \emph{bimodule category}, or the category of
 \emph{elements of the bimodule} $\El(\La^*)$. Recall \cite{dr01,bdf} that  
\begin{itemize}
	\item the \emph{objects} of $\El(\La^*)$ are just the elements of 
	$\hom_\La(P,\La^*\*_\La P)$,
 where $P$ runs through projective $\La$-modules;
  \item if $u\in\hom_\La(P,\La^*\*_\La P),\,v\in\hom_\La(P',\La^*\*_\La P')$,
  \emph{morphisms} $u\to v$ are homomorphisms $\al:P\to P'$ such that 
  $v\al=(1\*\al)u$.
\end{itemize}
  Thus isomorphisms
 $u\to v$, where $u,v\in\hom_\La(P,\La^*\*_\La P)$, are the elements $g\in\GL(P,\La)$
 such that $v=(1\*g)ug^{-1}$, which coincides with the adjoint action of
 $\GL(P,\La)$ on $\gl^*(P,\La)$. 
 
 Recall \cite{dr01} that the category $\El(\La^*)$ can be identified with 
 the category of representations of a Roiter box $\kL_\La=(A,V)$. Namely, 
 let $\kR$ be the radical of $\La$, $\La_{ij}=e_i\La e_j$,
 $\kR_{ij}=e_i\kR e_j$. Note that 
 $\La_{ij}=\kR_{ij}$ if $i\ne j$, while $\La_{ii}=\kR_{ii}\+\ko e_i$
 (since $\La$ is basic and $\ko$ is algebraically closed). 
 Choose a basis $\rB_0(j,i)$ of $\kR_{ij}$ and set
	\[
	\rB(j,i)=\begin{cases}
	\rB_0(j,i) &\text{if } i\ne j,\\
	\rB_0(i,i)\cup\{e_i\} &\text{if } i=j.
	\end{cases}
  \]
  (It is a basis of $\La_{ij}$.) Let $\rD(j,i)$ be the basis of $(\kR_{ij})^*$ 
  dual to $\rB_0(j,i)$, $\rB=\bigcup_{i,j}\rB(j,i),\
  \rD=\bigcup_{i,j}\rD(j,i)$,
  and $\ga(x,y,b)$ are the structure constants of the algebra $\La$,
  i.e. $xy=\sum_b\ga(x,y,b)b$ for $x,y,b\in\rB$. It implies that
 \begin{align*}
  &x^*y=\Sum_b\ga(y,b,x)b^*\ \text{ and }\ xy^*=\Sum_b\ga(b,x,y)b^*,\\
  &\ga(x,e_j,b)=\ga(e_j,x,b)=\de_{xb}\ \text{ for }\ x\in\rB(j,i).
 \end{align*}
  Then the set of solid arrows $j\to i$ in $\kL_\La$ is $\rB(j,i)$, while 
  the set of dotted arrows $j\dar i$ is $\rD(j,i)$.
  The differential is defined by the rules
  \begin{align*} 
  \pp(b)&= \Sum_{x,y} (\ga(b,x,y)xy^*-\ga(y,b,x)x^*y),\\
  \pp(b^*)&= \Sum_{x,y}\ga(x,y,b)x^*\*y^*,
  \end{align*}
  Especially,
  \[
	\pp(e_i)=\Sum_{x\in\rB(\cdot,i)} xx^* - \Sum_{y\in\rB(i,\cdot)} y^*y.
	\]
	Hence, setting $\aa=\set{e_i}$ and $\tb=b^*$ for $b\in\rB_0=\rB\sm\aa$,
	we get the following statement.
 
 	\begin{prop}\label{alg}
 	The box $\kL_\La$ is of \BT-type.
	\end{prop}
	
	By the way, it implies, due to Theorem~\ref{theo_main}\,(2), that the problem
	of description of orbits of the coadjoint action of $\GL(P,\La)$ for all $P$
	is wild, whenever the algebra $\La$ is not semisimple.

 Obviously, an element $\xi\in\gl^*(P,\La)$ is a brick \iff it has the trivial
  stabilizer: $\mathrm{Stab}_{\GL(P,\La)}(\xi)=\ko^\times$ 
  (the multiplicative group of the field $\ko$).

  \begin{coro}\label{lingr}
 If the set $S(P,\La)$ of bricks in $\gl^*(P,\La)$ is non-empty,
 all elements from this set belong to a single rational family. 
  In particular, $S(P,\La)/\GL(P,\La)\simeq\ko$.
  \end{coro}
  
  \begin{remk}
  One easily sees that if $S(P,\La)\ne\0$, it is open and dense in $\gl^*(P,\La)$.
  Unfortunately, in most cases it is empty. Nevertheless, there is at least 
  one case, when it is big indeed (see \cite{dr92,bdf}). It happens, when 
  $\La$ is a \emph{Dynkinian algebra}, that is an algebra derived equivalent to 
  the path algebra of a Dynkin quiver. Namely, if $\La$ is Dynkinian, the space 
 $\gl^*(P,\La)$ always contains an open dense subset $U$ such that all elements 
 $\xi\in U$ are \emph{semisimple}, i.e. direct sums of bricks that are \emph{mutually
 orthogonal}, that is every morphism between them is either zero or isomorphism. 
 Therefore, the stabilizer of a semisimple element is a product of full 
 linear groups over $\ko$.
  \end{remk}


\end{document}